\documentclass[]{amsart}
\usepackage[latin1]{inputenc}
\usepackage[english]{babel}
\usepackage{amssymb}
\usepackage{amsmath}
\usepackage{amsthm}
\usepackage{latexsym}
\usepackage{cite}
\newtheorem{definition}{Definition}[section]
\newtheorem{satz}{Theorem}

\newtheorem{lemma}[definition]{Lemma}

\theoremstyle{definition}
\newtheorem*{rems}{Remarks}
\renewcommand{\phi}{\varphi}
\DeclareMathOperator{\vol}{vol}

\newcommand{\skalar}[2]{\left\langle #1, #2 \right\rangle}
\newcommand{\norm}[1]{\left\Vert #1 \right\Vert}
\newcommand{\I}{\mathrm{i}}
\newcommand{\D}{\mathrm{d}}
\newcommand{\R}{\mathbb R}
\begin{document}
\title{Hyperplane Sections of Cylinders}
\author{Hauke Dirksen}
\address{Department of Mathematics, Kiel University}
\email{hauke.dirksen(at)gmx.de}
\date{February 25, 2016}
\begin{abstract}
We provide a formula to compute the volume of the intersection of a generalized cylinder with a hyperplane. Then we prove an integral inequality involving Bessel functions similar to Keith Ball's well-known inequality. Using this inequality we obtain upper bounds for the section volume. For large radius of the cylinder we determine the maximal section.
\end{abstract}
\maketitle
\section{Introduction}
The study of sections of certain convex bodies has a long history. The first formula for sections of the cube with a hyperplane dates back to Laplace 1812. The first results on bounds for the volume were found by D. Hensley \cite{Hensley1979} and K. Ball\cite{Ball1986}. The upper bound for the cube leads to a simple counterexample to the Busemann-Petty-Problem. So the study of hyperplane sections is linked with other problems in convex geometry. 
Many different convex bodies have been investigated. For example, $\ell_p$-balls in \cite{Meyer1988} and \cite{Koldobsky2005}, complex cubes in \cite{Oleszkiewicz2000}; also non-central sections in \cite{Moody2013} as well as taking other than Lebesgue measures in \cite{Koenig2013d} have been investigated. 
In this paper we deal with generalized cylinders.

Throughout this paper we use the following \emph{notations}: The Euclidean norm is denoted by $\norm{x}$, the standard scalar product by $\skalar{x}{y}$. 
For $a \in \R^{n}$ with $\norm{a}=1$ and $t \in \R$, let $H_a^t:= \{x \in \R^{n} \mid \skalar{a}{x}=t\} = H_a + t\cdot a$ be a translated hyperplane, especially $H_a:=H_a^0$. 
If $H$ is a $k$-dimensional (affine) subspace and  $A\subset H$, the $k$-volume of $A$ is the standard induced Lebesgue volume of the subspace, denoted by $\vol_k(A)$.
The characteristic function of a set $A$ is denoted by $\chi_A$.

The \emph{normalized} Bessel function of order $\nu$ is given by 
\[
j_{\nu}(s):=2^{\nu}\Gamma(\nu+1) \frac{J_{\nu}(s)}{s^{\nu}} \ \text{ for } s>0 \text{ and } j_{\nu}(0):=1,
\] 
where $J_{\nu}$ is the Bessel function of order $\nu$. The normalized Bessel function $j_{\nu}$ is continuous in $0$. A classical introduction to Bessel functions is \cite{Watson1966}. 

We consider \emph{generalized cylinders}.  Let 
\[
Z:= \frac 1 2 B_{\infty}^n \times r B_2^m \subset \R^{n+m}
\]
for $r>0$,  $n,m\in\mathbb{N}$, where $B_{\infty}^n:=\left[-1,1\right]^n$ and $B_2^m:=\{x\in \R^m \mid \norm{x}_2\leq 1\}$. 
We are interested in the volume of central sections, i.e. in the quantity
\[
\vol_{n+m-1} (H_a\cap Z)
\]
for $a \in \R^{n+m},\norm{a}=1$. We may assume $a=(a_1,\dots, a_n, a_{n+1},0,\dots,0)$, with $a_1, \dots, a_{n+1} \geq 0$, since $Z$ is rotationally symmetric with respect to the coordinates $n+1,\dots, n+m$ and symmetric with respect to the origin. 

Our first result, proved in Section \ref{sec:cylinder formula} by the classical Fourier analytic method, is a volume formula:
\begin{satz}\label{thm: cylinder formula}
For the cylinder $Z\subset \R^{n+m}$, with $m,n\in \mathbb{N}$, $r>0$, and a normal vector $a\in \R^{n+m}$ the volume of the hyperplane section $H_a\cap Z$ is given by
\begin{equation*}\label{eq:formula cylinder}
\vol_{n+m-1}(H_a\cap Z)=
r^m \frac{\pi^{\frac{m}{2}-1}}{\Gamma \left(\frac{m}{2}+1\right)}
					\int_{0}^{\infty} \prod_{j=1}^n \frac{\sin(\frac{a_j s}{2})}{\frac{a_j s}{2}} 	\cdot	
					j_{\frac{m}{2}} (a_{n+1}rs) \D s.
\end{equation*}
\end{satz}
Note that $j_{\frac m 2 }(s)=\frac{\sin s}{s}$, so for $m=1$ we get the formula for the cube.

Using H\"older's inequality in order to get an upper bound on the section volume is also a classical method. In Section \ref{sec:cylinder estimates} we follow this approach and find estimates on the volume. 
\begin{satz}\label{thm:bound cylinder}
Let $n>1,m>1$ and $r>0$.  Then for all $a\in\R^{n+m}$ with $\norm{a}=1$,
\begin{align}\label{eq:boundcylinder}
\vol_{n+m-1} (H_a\cap Z) &\leq  \begin{cases}
 r^{m} \frac{\pi^{\frac{m}{2}}}{\Gamma \left(\frac{m}{2}  +1\right)} \cdot \sqrt{2},   &r\geq \frac{\Gamma \left(\frac{m}{2}+1\right)}{\Gamma\left(\frac{m}{2}+\frac 1 2\right)}\frac{1}{\sqrt{\pi}}\\
r^{m-1} \frac{\pi^{\frac{m-1}{2}}}{\Gamma \left(\frac{m-1}{2}  +1\right)} \cdot \sqrt{2} , \quad &r <\frac{\Gamma \left(\frac{m}{2}+1\right)}{\Gamma\left(\frac{m}{2}+\frac 1 2\right)}\frac{1}{\sqrt{\pi}}.
\end{cases}
\end{align}
For $r\geq \frac{\Gamma \left(\frac{m}{2}+1\right)}{\Gamma\left(\frac{m}{2}+\frac 1 2\right)}\frac{1}{\sqrt{\pi}}$, the bound is attained for $a=\left(\frac{1}{\sqrt{2}},\frac{1}{\sqrt{2}},0,\dots,0\right)$. 
\end{satz}
For the three-dimensional case ($n=2$, $m=1$) real calculus suffices to characterize the maximal section. Note that the intersecting hyperplane can be described by one variable. We find
\begin{satz}\label{thm:3d cylinder}
Let $Z:=[-\frac 1 2,\frac 1 2 ] \times r\cdot B_2^2$. Depending on $r$ we have:\\
(i) For $r>\frac{1}{2\sqrt{3}}$ a section orthogonal to $(\sqrt{1-\alpha^2},\alpha,0)$ for some  $\alpha \in \left(\sqrt{\frac{1}{4r^2+1}},1\right)$ is maximal. So the maximal section is a truncated ellipse.\\
(ii) For $r\leq \frac{1}{2\sqrt{3}}$ the section orthogonal to $(0,1,0)$ is maximal. So the maximal section is a rectangle.
\end{satz}
In Section \ref{sec:cylinder inequality} we prove the main integral inequality, which is also interesting by itself. For the proof we use three slightly different approaches, depending on $m$.  The inequality states
\begin{satz}\label{thm_bessel_int} 
For all $m\in \mathbb{N}, m\geq 2$, and $p\in \R,p\geq 2$, we have
\[
\mathcal{J}_m(p):=\sqrt{p} \int_0^{\infty} \left| j_{\frac{m}{2}}(s)\right|^p  \D s\leq \sqrt{\pi}\sqrt{\frac{m}{2}+1}
\]
and $\lim_{p\to\infty}{\mathcal{J}}_m(p)= \sqrt{\pi}\sqrt{\frac{m}{2}+1}$.
\end{satz}
Recall Keith Ball's inequality \cite{Ball1986}. 
 For $p\geq 2$,
\begin{equation}\label{thm_ball}
\mathcal{J}_{1}(p)= \sqrt{ p } \int_0^{\infty} \left| \frac{\sin(u)}{u} \right|^{p} \D u 
\leq \frac{\pi}{\sqrt{2}},
\end{equation}
and $\lim_{p\to\infty} \mathcal{J}_1(p) =\sqrt{\frac{3}{2}\pi}< \frac {\pi}{\sqrt{2}}$.
\section{Volume formula}\label{sec:cylinder formula} We apply the standard method.
\begin{proof}[Proof of Theorem \ref{thm: cylinder formula}]
Define
$A(a, t):=\vol_{n+m-1} (H_a^t\cap Z)$ for $ a \in \R^{n+m}$, $\norm{a}=1$ and $t \geq 0$;
in particular $A(a):=A(a,0)$. 
We apply the Fourier transformation and the inversion formula to the function $t\mapsto A(a,t)$. 
With Fubini's theorem and the well-known integrals
\[
\int_{B^n_2} \exp \left(-\I s  \skalar{x}{a}\right) \D x 
=\frac{\pi^{\frac{n}{2}} }{\Gamma \left(\frac{n}{2}+1\right)} j_{\frac{n}{2}} (s\norm{a}), \quad \text{ for }s>0, a\in \R^n
\]
and
\[
\int_{\left[-\frac 1 2,\frac 1 2\right]^n} \exp \left(- \I s \skalar{x}{a}\right) \D x
=\prod_{j=1}^{n}\frac{\sin \left(\frac{a_js}{2}\right)}{\frac{a_j s}{2}},\quad \text{ for }s\in \R, a\in \R^{n}
\]
we have
\begin{align*}
					&\phantom{=} \left(2\pi\right)^{\frac 1 2}\hat{A}(a,s) \\
					&= \int_{\R} A(a,t) \exp \left(-\I st\right) \D t \\
		       &= \int_{\R}  \int_{\skalar{x}{a}=t} \chi_{\left[-\frac 1 2,\frac 1 2\right]^n} ((x_1,\dots,x_n)) \ 		 
		   		\chi_{r B_2^m}((x_{n+1},\dots,x_{n+m})) \exp \left(-\I st\right)\D x \D t \\
		      &= \int_{\R^{n+m}} \chi_{\left[-\frac 1 2,\frac 1 2\right]^n} ((x_1,\dots,x_n))		 \
		   		\chi_{r B_2^m}((x_{n+1},\dots,x_{n+m})) \exp \left(-\I s\skalar{x}{a}\right)\D x\\	
		   	  &=\int \limits_{\left[-\frac 1 2,\frac 1 2\right]^n} \exp\Big(-\I s \sum_{j=1}^n a_j x_j\Big) \D (x_1,\dots,x_n)   
		   		\int \limits_{r B_2^m} \exp\Big(-\I s  \sum_{j=n+1}^{n+m}a_{j} x_{j}\Big)\D x\\ 
					&= \prod_{j=1}^n \frac{\sin(\frac{a_j s}{2})}{\frac{a_j s}{2}} 
								\cdot	
					r^m \frac{\pi^{\frac{m}{2}}}{\Gamma \left(\frac{m}{2}+1\right)}
 j_{\frac{m}{2}} (rsa_{n+1}). 
\end{align*}
Finally, by the Fourier inversion formula we get the formula stated in \mbox{Theorem \ref{thm: cylinder formula}}.
\end{proof}

For the three-dimensional cylinder, i.e. $n=1$ and $m=2$, using an equation from \cite[6.693 (4), p. 720]{Gradshteyn2007} we get
\begin{lemma}\label{lemma_cylinder}
Let $Z$ be the three-dimensional cylinder with radius $r>0$. For $\alpha \in [0,1]$ let $a=\left(\sqrt{1-\alpha^2},\alpha,0\right)$. Then the volume, i.e.\ the area, of the section $H_a\cap Z$ is given by the function $A\colon[0,1]\to \R,$
\begin{equation*}
	A(\alpha)=
	\begin{cases}
		\pi r\frac{r}{\sqrt{1-\alpha^2}}
				&\text{ for } \ 0 \leq \alpha\leq \frac{1}{\sqrt{1+4r^2}}\\
		 \frac{r}{\alpha} \sqrt{1-\frac{1-\alpha^2}{4\alpha^2r^2}} + \frac{2r^2}{\sqrt{1-\alpha^2}} \arcsin\left(\frac{\sqrt{1-\alpha^2}}{2\alpha r}\right)
				&\text{ for } \ \frac{1}{\sqrt{1+4r^2}} < \alpha <1 \\
		2r &\text{ for }   \ \alpha =1.
	\end{cases}
\end{equation*}
\end{lemma}
The three cases correspond to the geometric shape of the section, namely an ellipse resp. a disk, a truncated ellipse and a rectangle. Clearly, this formula can also be obtained by elementary geometric considerations. 
\section{Volume estimates}\label{sec:cylinder estimates}
\subsection{Three-dimensional case} The three-dimensional case can be treated by real calculus. 
\begin{proof}[Proof of Theorem \ref{thm:3d cylinder}]
The function $A$ from Lemma \ref{lemma_cylinder} defined on the closed interval $[0,1]$ is differentiable. 
Let $\alpha^*:=\frac{1}{\sqrt{1+4r^2}}$. 
For $0< \alpha < \alpha^*$ we have $A^{\prime}(\alpha)=\frac{\pi r^2 \alpha}{(1-\alpha^2)^{\frac{3}{2}}}$. This  is larger than $0$ for all $r>0$. For the left derivative in $\alpha^*$ we get $ A^{\prime}_{-}(\alpha^*)=\frac{\pi(1+4r^2)}{8r}$.

For $\alpha^*<\alpha < 1$ we find 
\begin{align}\label{eq:derivative_cylinder}
A^{\prime}(\alpha)&=
\frac{1}{4r\alpha^4}\frac{1}{\sqrt{1+\frac{1}{4r^2}-\frac{1}{4r^2\alpha^2}}}
-\frac{r}{\alpha^2}\sqrt{1+\frac{1}{4r^2}-\frac{1}{4r^2\alpha^2}}   \\  \nonumber
&\phantom{=} +\frac{2 \alpha r^2}{(1-\alpha^2)^{\frac{3}{2}}}\arcsin\left(\frac{\sqrt{1-\alpha^2}}{2\alpha r}\right)
- \frac{r}{\alpha^2(1-\alpha^2)}\frac{1}{\sqrt{1+\frac{1}{4r^2}-\frac{1}{4r^2\alpha^2}}}.
\end{align}
Compute the limit of (\ref{eq:derivative_cylinder}) for $\alpha \to \alpha^*$, $\alpha>\alpha^*$. 
Note that for $\alpha =\alpha^*$ we have $\sqrt{1+\frac{1}{4r^2}-\frac{1}{4r^2\alpha^2}}=0$. The sum of the first and the last summand of (\ref{eq:derivative_cylinder}) tends to $0$ by L'H\^opital's rule. The second summand tends to $0$ as well. The third summand tends to $\frac{\pi(1+4r^2)}{8r}$, which coincides with the left derivative in $\alpha^*$. So $A$ is differentiable in $(0,1)$ with $A^{\prime}(\alpha^*)=\frac{\pi(1+4r^2)}{8r}>0$ for all $r>0$.

In particular, $A^{\prime}$ is positive on $(0,\alpha^{*}]$. So $A$ is maximal for some $\alpha \in (\alpha^*,1]$. The maximum is attained for some $\alpha < 1$ if and only if $A^{\prime}(\alpha)$ has a zero in $(\alpha^*,1)$. Otherwise the function $A$ is monotonously increasing from $0$ to $1$ and attains the maximum for $\alpha=1$.

For $\alpha \in (\alpha^*,1)$ the equation $A'(\alpha)=0$ is equivalent to the following equation:
\begin{align*}
\frac{2 \alpha r^2}{(1-\alpha^2)^{\frac{3}{2}}}\arcsin\left(\frac{\sqrt{1-\alpha^2}}{2\alpha r}\right)
&=\frac{r}{\alpha^2(1-\alpha^2)}\frac{1}{\sqrt{1+\frac{1}{4r^2}-\frac{1}{4r^2\alpha^2}}} \\
&\phantom{=}+\frac{r}{\alpha^2}\sqrt{1+\frac{1}{4r^2}-\frac{1}{4r^2\alpha^2}}   \\  
&\phantom{=}-\frac{1}{4r\alpha^4}\frac{1}{\sqrt{1+\frac{1}{4r^2}-\frac{1}{4r^2\alpha^2}}}.
\end{align*}
Multiplying this by $\frac{1-\alpha^2}{r}$ and adding the first and the third summand on the right-hand side this simplifies to
\begin{align}\label{eq:cond_cylinder}
\frac{\arcsin\left(\frac{\sqrt{1-\alpha^2}}{2\alpha r}\right)}{\frac{\sqrt{1-\alpha^2}}{2\alpha r}} 
				=\frac{2-\alpha^2}{\alpha^3}\sqrt{\alpha^2+\frac{\alpha^2}{4r^2}-\frac{1}{4r^2}}.
\end{align}
Set $x:=\frac{\sqrt{1-\alpha^2}}{2\alpha r}$, then $\alpha=\frac{1}{\sqrt{1+4r^2x^2}}$. Equation (\ref{eq:cond_cylinder}) reads as
\begin{align}\label{eq:cond_equiv}
\frac{\arcsin(x)}{x} 
				=(1+8r^2x^2)\sqrt{1-x^2}.
\end{align}
So $A^{\prime}(\alpha)=0$ for some $\alpha \in (\alpha^*,1)$ is equivalent to (\ref{eq:cond_equiv}) for some $x\in (0,1)$.
Estimating both sides of equation (\ref{eq:cond_equiv}) using Taylor's theorem we find that $A^{\prime}$ has a zero smaller than $1$ if and only if $r>\frac 1 {2\sqrt{3}}$. 
\end{proof}
\subsection{General dimension}
The first step is the application of H\"older's inequality.
\begin{lemma}\label{lem_hoelder_cylinder}
Let $a\in\R^{n+m}$ be a normal vector. Then
\begin{equation*} \label{eqn_hoelder}
\vol_{n+m-1}(H_a\cap Z) \leq r^m \frac{\pi^{\frac{m}{2}-1}}{\Gamma \left(\frac{m}{2}+1\right)}
\prod_{j=1}^{n} \bigg(2 \mathcal{J}_{1}\Big(\frac{1}{a_j^2}\Big)\bigg)^{a_j^2} \ 
\bigg(\frac{1}{r}\mathcal{J}_m\Big(\frac{1}{a_{n+1}^2}\Big)\bigg)^{a_{n+1}^2} ,	
\end{equation*}
where 
\begin{align*}
\mathcal{J}_m(p)&:= \sqrt{ p } \left(\int_0^{\infty} \left| j_{\frac{m}{2}}(u) \right|^{p}du \right).
\end{align*}

\begin{proof}
We apply H\"older's inequality to the formula from Theorem \ref{thm: cylinder formula} and then substitute $u=\frac{a_j s}{2}$ resp. $u=a_{n+1}rs$:
\begin{align*}
&\phantom{=} \vol_{n+m-1}(H_a\cap Z) \nonumber \\
&\leq r^{m} \frac{\pi^{\frac{m}{2}-1}}{\Gamma (\frac{m}{2}  +1)} 
		\prod_{j=1}^n \bigg( \frac{2}{a_j} \int\limits_0^{\infty} \left| \frac{\sin u}{u}\right|^{\frac{1}{a_j^2}} du\bigg)^{a_j^2} \
	\bigg( \frac{1}{r a_{n+1}} \int\limits_0^{\infty} \left| j_{\frac{m}{2}}(u)\right|^{\frac{1}{a_{n+1}^2}} du \bigg)^{a_{n+1}^2} \nonumber \\
&= r^{m} \frac{\pi^{\frac{m}{2}-1}}{\Gamma (\frac{m}{2}  +1)} 
			\prod_{j=1}^n\bigg(2\mathcal{J}_1\Big(\frac{1}{a_j^2}\Big)\bigg)^{a_j^2}
			\ \bigg(\frac{1}{r}\mathcal{J}_m\left(\frac{1}{a_{n+1}^2}\right)\bigg)^{a_{n+1}^2}. 	\qedhere
\end{align*}
\end{proof}
\end{lemma}

\begin{proof}[Proof of Theorem \ref{thm:bound cylinder}]
The integral inequality from Theorem \ref{thm_bessel_int} and also Ball's inequality (\ref{thm_ball}) may only be used if all coordinates of a are smaller than $\frac{1}{\sqrt{2}}$. If there is a coordinate larger than $\frac{1}{\sqrt{2}}$, we use a different estimate that is also used in Ball's proof, for example\cite{Ball1986}.

\textbf{Case 1:} Let $|a_j|\leq \frac{1}{\sqrt{2}}$ for all $j=1,\dots,n+1$, so there is no dominating coordinate.
We apply the integral inequality (\ref{thm_ball}) and the one from Theorem \ref{thm_bessel_int} to Lemma \ref{lem_hoelder_cylinder}. 
For the third inequality, note that $\frac{1}{r}\sqrt{\pi}\sqrt{\frac m 2 +1} < \sqrt{2}\pi$ if and only if $r> \frac{\sqrt{\frac m 2 +1}}{\sqrt{2\pi}}$,  so 
\begin{align*}
\vol_{n+m-1} (H_a\cap Z) &\leq r^{m} \frac{\pi^{\frac{m}{2}-1}}{\Gamma (\frac{m}{2}  +1)} 
			\prod_{j=1}^n\left(\sqrt{2}\pi \right)^{a_j^2}
			\cdot \left(\frac{1}{r}\mathcal J_m\left(\frac{1}{a_{n+1}^2}\right)\right)^{a_{n+1}^2}\\
&\leq \frac{\pi^{\frac{m}{2}-1}}{\Gamma (\frac{m}{2}  +1)}  \left(\sqrt{2}\pi\right)^{\sum_{j=1}^n a_j^2} \left(\frac{1}{r}\sqrt{\pi}\sqrt{\frac m 2 +1}\right)^{a_{n+1}^2} \\
&\leq \begin{cases}
 r^{m} \frac{\pi^{\frac{m}{2}}}{\Gamma (\frac{m}{2}  +1)}  \sqrt{2} ,   &r> \frac{\sqrt{\frac m 2 +1}}{\sqrt{2\pi}}\\
 r^{m-1} \frac{\pi^{\frac{m-1}{2}}}{\Gamma (\frac{m}{2}  +1)}  \sqrt{\frac{m}{2}+1}, \quad &r \leq \frac{\sqrt{\frac m 2 +1}}{\sqrt{2\pi}}.
\end{cases}
\end{align*}

\textbf{Case 2:} Let $|a_j|> \frac{1}{\sqrt{2}}$ for some $j=1,\dots,n$. 
Let $P$ be the orthogonal projection onto the hyperplane $\{x_j=0\}$. 
Since $P(H\cap Z)\subset P(Z)$, we have $\vol(P(H\cap Z))\leq \vol (P(Z))$. 
The projected cylinder $P(Z)$ is isomorphic to $\frac 1 2 B_{\infty}^{n-1} \times rB_2^m$, so the volume can be computed elementary. 
Furthermore, 
\[
\vol_{n+m-1}(H_a\cap Z)=\frac{1}{|a_j|} \vol_{n+m-1}(P(H_a\cap Z)).
\]
Therefore
\begin{align*}
\vol_{n+m-1} (H_a\cap Z) &< \sqrt{2} \vol_{n+m-1}(P(Z))\\
&=\sqrt{2} r^m\frac{\pi^{\frac{m}{2}}}{\Gamma\left(\frac{m}{2}+1\right)}.
\end{align*}

\textbf{Case 3:} Let $|a_j|> \frac{1}{\sqrt{2}}$ for $j=n+1$.
We consider the orthogonal projection onto $\{x_{n+1}=0\}$. Now  $P(Z)$ is isomorphic to $\frac 1 2 B_{\infty}^n \times B_2^{m-1}$. 
By the same argument as in case 2,
\begin{align*}
\vol_{n+m-1} (H_a\cap Z)
&<\sqrt{2} r^{m-1}\frac{\pi^{\frac{m-1}{2}}}{\Gamma(\frac{m-1}{2}+1)}.\label{eq:cylinder subcase3}
\end{align*}

We summarize the estimates. Note that by Lemma \ref{lem:quotient gamma} for $m\geq 2$:
\begin{equation}\label{eq:div radii}
\frac{\Gamma \left(\frac{m}{2}+1\right)}{\Gamma\left(\frac{m}{2}+\frac 1 2\right)}\frac{1}{\sqrt{\pi}} 
>\frac{\sqrt{\frac m 2 +1}}{\sqrt{2\pi}}.
\end{equation}
Let $r\geq \frac{\Gamma \left(\frac{m}{2}+1\right)}{\Gamma\left(\frac{m}{2}+\frac 1 2\right)}\frac{1}{\sqrt{\pi}}$. Due to (\ref{eq:div radii}), also $r>\frac{\sqrt{\frac m 2 +1}}{\sqrt{2\pi}}$. 
So in all three cases, we have $\vol_{n+m-1}(H_a\cap Z)\leq r^{m} \frac{\pi^{\frac{m}{2}}}{\Gamma \left(\frac{m}{2}  +1\right)} \sqrt{2}$. 
This bound is attained for the normal vector $a=\left( \frac{1}{\sqrt{2}},\frac{1}{\sqrt{2}},0,\dots,0\right)$.

If $r<\frac{\Gamma \left(\frac{m}{2}+1\right)}{\Gamma\left(\frac{m}{2}+\frac 1 2\right)}\frac{1}{\sqrt{\pi}}$, then the bound from case 3 is the largest. 
\end{proof}

\begin{rems}
(i) We did not touch the question if the distinction of the cases in (\ref{eq:boundcylinder})   is sharp.
In Theorem \ref{thm:3d cylinder} the distinction of the cases is sharp. In this theorem, for $n=1$ and $m=2$ the critical radius would be equal to $\frac{4}{\pi^2}$, which is much larger than the critical radius $\frac{1}{2\sqrt{3}}$ from Theorem \ref{thm:3d cylinder}.

(ii) For the three-dimensional cylinder we found that a \emph{truncated ellipse} gives maximal volume for large $r$. For the generalized cylinder there is a different behavior. The volume-maximal section of the cylinder is the Cartesian product of the maximal section of the cube and a ball of dimension $m$. For example, for a $4$-dimensional cylinder, i.e. $n=2=m$, for large $r$ the maximal section is a three-dimensional cylinder of height $\sqrt{2}$ and radius $r$.

(iii) We conjecture that, if $r$ is sufficiently small, the section orthogonal to $a=(0,\dots,0,1,0,\dots,0)$ is maximal, where the $(n+1)$-th coordinate of $a$ is $1$. The volume of this section is equal to 
\[ \vol_n\left( \frac 1 2 B_{\infty}^n\right) \vol_{m-1}\left(rB_2^{m-1}\right)=r^{m-1}\frac{\pi^{\frac{m-1}{2}}}{\Gamma\left(\frac {m-1}{2} + 1\right)}.\]
Comparing this to the bound from (\ref{eq:boundcylinder}), there is an error of $\sqrt{2}$.\\
Numerical experiments suggest that for medium sized $r$, some non-standard direction is maximal.

(iv) Ball's and our inequality have a different behavior. This indicates why Theorem \ref{thm:bound cylinder} is not always sharp. Note that $\mathcal J_1(2)>\lim_{p\to\infty}{\mathcal{J}}_1(p)$ in contrast to $\mathcal J_m(2) \leq \lim_{p\to\infty}{\mathcal{J}}_m(p)$ for  $m\geq 2$. So for $m=1$, equality holds for $p=2$ in contrast to $m\geq 2$, where equality holds for $p=\infty$.

(v) As Theorem \ref{thm:3d cylinder} shows, there is a critical value of the radius that originates in the geometry of the cylinder. For generalized cylinders an additional distinction comes from the method, and this does not give the sharp geometric distinction as in Theorem \ref{thm:3d cylinder}.
\end{rems}

\section{Integral inequality}\label{sec:cylinder inequality}
Integral inequalities similar to  Theorem \ref{thm_bessel_int} and (\ref{thm_ball}) were established for complex cubes and for generalized cubes, see \cite{Oleszkiewicz2000} and \cite{Brzezinski2011}. Identifying $\mathbb{C}^n$ and $\mathbb{R}^{2n}$, hyperplane sections of the complex cube have real dimension $2n-2$.
The  integral inequality needed for this case is 
\begin{equation}\label{eq:oleskiewicz}
\sqrt{p} \int_0^{\infty} \left| j_{1}(s)\right|^p  s \ \D s\leq \frac 4 p, 
\end{equation} 
for $p\geq 2$. Note that compared to (\ref{thm_ball}) there is an additional factor $s$ in front of $\D s$. For generalized cubes one has to consider a similar integral with some higher power of $s$ in front of $\D s$.\\

We prove Theorem \ref{thm_bessel_int} by applying the following lemma due to Nazarov and Podkorytov \cite{Nazarov2000} . They used this lemma to simplify K. Ball's proof of \mbox{inequality (\ref{thm_ball})}. The oscillating behavior of the function  $\sin(s)/s$ is a main difficulty. By the Nazarov-Podkorytov lemma one avoids the oscillations by considering the distribution functions. These functions are decreasing. 

For a function $f:X \to \R_{\geq 0}$ on a measure space $(X,\mu)$, define the cumulative distribution function 
$F: \R_{>0} \to \R_{\geq 0}$ by 
\[
F(y):=\mu (\{x\in X \mid f(x)>y\}).
\]
\begin{lemma}[Nazarov-Podkorytov]\label{prop:nazarov}
Let $h$, $g$ be non-negative measurable functions on a measure space $(X,\mu)$. Let $H$, $G$ be their distribution functions. Assume that $H(y)$, $G(y)$ are finite for all $y>0$. Also assume that 
\begin{enumerate}
\item [(N1)] \label{cond_naz 1} there is some $y_0 >0 $ such that $G(y)\leq H(y)$ for all $y<y_0$ and $G(y)\geq H(y)$ for all $y>y_0$, i.e. the difference $G-H$ changes its sign exactly once from $-$ to $+$; 
\item [(N2)] \label{cond_naz 2} for some $p_0>0$: $\int_X h^{p_0} d\mu = \int_X g^{p_0}d\mu$. 
\end{enumerate}
Then 
\[
		\int_X h^{p} d\mu \leq \int_X g^{p}d\mu 
\] 
for all $p>p_0$ as long as the integrals exist.
\end{lemma}
\subsection{Technical estimates}
The proof of the integral inequality uses some technical estimates that we state here.
\begin{lemma}\label{lem:quotient gamma}
For $x\geq 2$ we have
\[
\frac{\Gamma (x)}{\Gamma\left(x-\frac{1}{2}\right)} >\frac{\sqrt{x}}{2}.
\]
\begin{proof}
We estimate the gamma functions by Stirling's formula:
\begin{align*}
\frac{\Gamma (x)}{\Gamma\left(x-\frac{1}{2}\right)} &\geq \left(\frac{x}{x-\frac 1 2}\right)^x \frac{x-\frac 1 2}{\sqrt{x}}\frac{1}{\exp(\frac 1 2)}\frac{1}{\exp(\frac{1}{24})} \\
&\geq  \frac{x-\frac 1 2}{\sqrt{x}}\frac{1}{\exp(\frac{1}{24})}\\
&=\left(\sqrt{x}-\frac{1}{2\sqrt{x}}\right)\exp\left(-\frac{1}{24}\right)
\end{align*}
Note that $\left(\frac{x}{x-\frac 1 2}\right)^x$ strictly decreases to $\exp(\frac 1 2)$.
Additionally $\left(\sqrt{x}-\frac{1}{2\sqrt{x}}\right)\exp(-\frac{1}{24})$ increases faster than $\frac{\sqrt{x}}{2}$ and the inequality holds for $x=2$.
\end{proof}
\end{lemma}

\begin{lemma}\label{lem:Gammafunctions}
Let $m\geq 5$. Then we have
\[
\frac{\Gamma \left(\frac{m}{2} +1\right)^2 \Gamma (m)}{\Gamma\left(\frac{m}{2}+\frac{1}{2}\right)^2\Gamma \left(m+\frac{1}{2} \right)} \leq \frac{m+2}{m+1}\frac{\sqrt{m}}{2}.
\]
\begin{proof}
Using Legendre's duplication formula, we find
\begin{align*}
\frac{\Gamma \left(\frac{m}{2} +1\right)^2 \Gamma (m)}{\Gamma\left(\frac{m}{2}+\frac{1}{2}\right)^2\Gamma \left(m+\frac{1}{2} \right)} 
&= 
\frac{\left(\frac m 2\right)^2\Gamma\left(\frac m 2\right)^2 \Gamma\left(\frac m 2\right)^2 \ \  \Gamma (m)\Gamma(m)}{\left[\Gamma\left( \frac m 2 + \frac 1 2 \right)^2 \Gamma\left(\frac m 2 \right)^2\right]\left[\Gamma \left(m  +\frac 1 2\right)\Gamma(m) \right]}\\
&=\frac {m^2}{4}\frac{\Gamma\left(\frac m 2 \right)^4 \Gamma(m)^2}{\left[\frac{\sqrt{\pi}}{2^{m-1}}\Gamma (m)\right]^2\frac{\sqrt{\pi}}{2^{2m-1}}\Gamma(2m)}\\
&=2^{4m-5}\frac{m^2}{\pi^{\frac 3 2}} \frac{\Gamma\left(\frac m 2\right)^4}{\Gamma (2m)}.
\end{align*}
So we need to show
\begin{align*}
q(m):=\frac{2^{4m-4}}{\pi^{\frac 3 2}}\frac{m^{\frac 3 2}(m+1)}{m+2} \frac{\Gamma\left(\frac m 2\right)^4}{\Gamma (2m)} \leq 1.
\end{align*}
By application of Stirling's formula we find
\begin{align*}
q(m) &\leq \frac{2^{4m-4}}{\pi^{\frac 3 2}}\frac{m^{\frac 3 2}(m+1)}{m+2} \frac{\left(\sqrt{2\pi}\left(\frac m 2 \right)^{\frac m 2 - \frac 1 2 }\exp\left(-\frac m 2 \right) \exp\left(\frac{1}{6m}\right)\right)^4}{\sqrt{2\pi}(2m)^{2m-\frac 1 2}\exp(-2m)}\\
&=\frac{m+1}{m+2}\exp \left(\frac 2 {3m}\right)=: \tilde q (m).
\end{align*}
As a function on $\R_{\geq 0}$, the derivative of $\tilde q(m)$ only has a zero in $m=3+\sqrt{13}> 6$. Note that $\tilde q(5)=\frac 6 7 \exp\left(\frac 2 {15}\right)<1$ and $\tilde q(6)=\frac 7 8 \exp\left(\frac 2 {18}\right)<\tilde q(5)$.
Obviously $\tilde q(m) \to 1$ for $m\to \infty$. Therefore $\tilde q(m)$ is increasing for $m\geq 7$. 
This proves $\tilde q(m)<1$ for all $m\geq 5$.
\end{proof}
\end{lemma}

A bound for the absolute value of Bessel functions follows from \cite[(8.479)]{Gradshteyn2007}.
For the normalized Bessel functions this reads as:
\begin{lemma}\label{eq:estimatebessellarge}
Let $m\in \mathbb{N}$ and $s> \frac{m}{2}$. Then
\begin{equation*}
|j_{\frac{m}{2}}(s)|\leq \frac{2^{\frac{m+1}{2}}\Gamma(\frac{m}{2}+1)}{\sqrt{\pi}} \frac{1}{\left(s^2-\frac{m^2}{4}\right)^{\frac{1}{4}}} \frac{1}{s^{\frac{m}{2}}}.
\end{equation*}
\end{lemma}

More elaborated estimates were used in several contexts. We collect a few results that we need later.
\begin{lemma}\label{lem:bessel small}
Let $m\geq 2$ and $s\in [0,\frac{m}{2}+3]$ resp. let $m=1$ and $s\in [0,3.38]$. Then
\begin{equation*}\label{eq:estimatebesselsmall}
|j_{\frac{m}{2}}(s)|\leq \exp \left(-\frac{s^2}{2m+4}-\frac{s^4}{4(m^2+2m+4)(m+4)}\right).
\end{equation*}
\begin{proof}
This is found in \cite[p. 19]{Koenig2001}.
\end{proof}
\end{lemma}
\begin{lemma}\label{lem:bessel exp}
Let $m\geq 5$ and $s\in [0,m]$. Then we have
\[
|j_{\frac{m}{2}}(s)|\leq \exp \left(-\frac{s^2}{2m+4}\right).
\]
\begin{proof}
Let $m=5$ or $m=6$. Then the inequality follows directly from Lemma \ref{lem:bessel small}, since $\frac m 2 + 3 \geq m$. The same lemma shows the inequality for $m\geq 7$ and $s \in [0, \frac m 2  +3]$.
In \cite[Lemma 3.17]{Brzezinski2011} it is proved that for all $m\geq 7$ and $s \in [\frac m 2 +3,m]$ the claimed inequality also holds. Brzezinski's proof uses the estimate from Lemma \ref{eq:estimatebessellarge}.
\end{proof}
\end{lemma}

\begin{lemma}\label{lem:besselpointwiseestimate}
Let $m\in \mathbb{N}$ and $s\geq \frac{m}{2}+3$. Then 
\[
|j_{\frac{m}{2}}(s)|\leq 2^{\frac{m+1}{2}} \frac{\Gamma(\frac{m}{2}+1)}{\sqrt{\pi}} \frac{\sqrt{m+6}}{\sqrt[4]{12m+36}} \frac{1}{s^{\frac{m+1}{2}}}.
\]
\begin{proof}
For $s\geq \frac{m}{2}+3$ we have
\begin{align*}
\left(s^2-\frac{m^2}{4}\right)^{-\frac{1}{4}}s^{-\frac m 2}
&= 		\left(1-\frac{m^2}{4s^2}\right)^{-\frac{1}{4}}s^{-\frac {m+1} 2} \\
&\leq	\frac{\sqrt{m+6}}{\sqrt[4]{12m+36}}.
\end{align*}
The estimate follows together with Lemma \ref{eq:estimatebessellarge}.
\end{proof}
\end{lemma}

We also need a lower bound on $\left|j_{\frac{m}{2}}(\cdot)\right|$.
\begin{lemma}\label{lem:bessel lower bound}
For all $m\in\mathbb{N}$ and $s \in [0,1]$ we have
\begin{eqnarray*}
\left|j_{\frac{m}{2}}(s)\right|\geq \exp \left(-\frac{s^2}{2m+4}-s^4\right).
\end{eqnarray*}
\begin{proof}
This is found in \cite[Lemma 3.5, part 2]{Brzezinski2011}. The estimate there is even stronger.
\end{proof}
\end{lemma}

\begin{lemma}\label{lem:estimateintegral}
For $p>0$ and $m \in \mathbb{N}$ we have
\[
\int_0^{\infty} e^{-\frac{ps^2}{2m+4}-\frac{ps^4}{4(m+2)^2(m+4)}} \D s\leq \frac{1}{\sqrt{p}} \sqrt{\frac{m}{2}+1}\sqrt{\pi}\left(1-\frac{3}{4}\frac{1}{p(m+4)}+\frac{105}{16}\frac{1}{2p^2(m+4)^2} \right).
\]
\begin{proof}
By substituting $u:=\frac{ps^2}{2m+4}$ we get
\begin{align*}
\int_0^{\infty} e^{-\frac{ps^2}{2m+4}-\frac{ps^4}{4(m+2)^2(m+4)}} \D s
			&=
\frac{1}{2}\sqrt{\frac{2m+4}{p}}    \int_0^{\infty} e^{-u }e^{-\frac{u^2}{p(m+4)}}u^{-\frac 1 2} \D u.
\end{align*}
Then we estimate the exponential function $\exp \left( -\frac{u^2}{p(m+4)}\right)$ by the first three summands of its series expansion. 
\end{proof}
\end{lemma}

\subsection{The limit of the integral}
We prove the asymptotic result of the integral inequality from Theorem \ref{thm_bessel_int}. Using Lemmas \ref{lem:bessel small} and \ref{lem:besselpointwiseestimate}, we estimate
\begin{align*}
&\sqrt{p} \int_0^{\infty} \left|j_{\frac{m}{2}}(s)\right|^p \D s \\
				\leq  &\sqrt{p} \int_0^{\frac{m}{2}+3} \exp \left(-\frac{s^2}{2m+4}\right)^p \D s \\
					&+ \sqrt{p} \ \left(2^{\frac{m+1}{2}} \frac{\Gamma(\frac{m}{2}+1)}{\sqrt{\pi}}\right)^p \left(\frac{\sqrt{m+6}}{\sqrt[4]{12m+36}}\right)^p  \int_{\frac{m}{2}+3}^{\infty} s^{-p\frac{m+1}{2}}\D s\\
				= &\sqrt{p}\int_0^{\infty} \exp \left(-\frac{ps^2}{2m+4}\right) \D s  \\ 
					&+ \sqrt{p} \ \left(2^{\frac{m+1}{2}} \frac{\Gamma(\frac{m}{2}+1)}{\sqrt{\pi}} \right)^p\left(\frac{\sqrt{m+6}}{\sqrt[4]{12m+36}}\right)^p  \frac{1}{p\frac{m+1}{2}-1} \left(\frac{m}{2}+3\right)^{-p\frac{m+1}{2}+1}.		
\end{align*}
For $p \to \infty$, the first summand tends to  $\sqrt{\pi}\sqrt{\frac{m}{2}+1}$ since $\int_0^{\infty} \exp(-x^2/K)\D x=\sqrt{K\pi}/2$  for $K>0$. 
Comparing the exponents, the second summand tends to $0$ for $p \to \infty$.

On the other hand, using Lemma \ref{lem:bessel lower bound}, by the substitution $u=\sqrt{p}s$ and by the series expansion of the exponential function we have
\begin{align*}
			\sqrt{p} \int_0^{\infty} \left| j_{\frac{m}{2}}(s)\right|^p \D s  
&\geq	\sqrt{p} \int_0^{1} \exp\left(-\frac{ps^2}{2m+4}-ps^4\right) \D s \\
&=			    \int_0^{\sqrt{p}} \exp\left(-\frac{u^2}{2m+4}-\frac{u^4}{p}\right) \D u \\
&\geq			 \int_0^{\sqrt{p}} \exp\left(-\frac{u^2}{2m+4}\right)\left(1-\frac{u^4}{p}\right) \D u \\
&\geq			  \int_0^{\sqrt{p}} \exp\left(-\frac{u^2}{2m+4}\right)\D u - \frac{1}{p}\int_{0}^{\sqrt{p}}u^4\exp\left(-\frac{u^2}{2m+4}\right)\D u \\
&\geq			 \int_0^{\sqrt{p}} \exp\left(-\frac{u^2}{2m+4}\right)\D u - \frac{1}{p}\int_{0}^{p}u\exp \left(-\frac{u}{2m+4}\right)\D u.
\end{align*}
For $p \to \infty$, we observe that the first summand again tends to $\sqrt{\pi}\sqrt{\frac{m}{2}+1}$, and the second summand vanishes since $\int_0^{\infty}x\exp(-x)\D x=1$.  By the sandwich lemma we have found the limit as claimed in Theorem \ref{thm_bessel_int}.

\subsection{The case $m=2$}
For $m=2$ the integral inequality from Theorem \ref{thm_bessel_int} is similar to Oleskiewicz's and Pe{\l}czy{\'n}ski's inequality to estimate the section volume of complex cubes, see (\ref{eq:oleskiewicz}). They used a different technique than we do. We use the Nazarov-Podkorytov lemma. This proof is a modification of an unpublished proof of Oleskiewicz's and Pe{\l}czy{\'n}ski's inequality by H. K\"onig \cite{Koenig2014}. 

We apply the Nazarov-Podkorytov lemma \ref{prop:nazarov} to the functions 
\[
		h(s):=\left|j_{1}(s)\right| = \left|\frac{2 J_1(s)}{s}\right| \text{ and } \ 
		g(s):=\exp\left(-\frac{s^2}{8}\right).
\]
By $H$ resp. $G$ we denote the distribution functions with respect to the Lebesgue measure $\lambda$ on $\R_{\geq 0}$. We check the two conditions of Lemma \ref{prop:nazarov}.

\subsubsection{Condition (N2)}
 Independently of $p$ we have \[\sqrt{p}\int_0^{\infty} g(s) ^{p} \D s = \sqrt{2 \pi}.\]
For $p=2$, we evaluate the other integral explicitly, using \cite[p. 403]{Watson1966}: \[\sqrt{2}\int_0^{\infty} h(s)^2 \D s=\frac{8\sqrt{2}}{3\pi}<\sqrt{2\pi}.\] 
By \cite[(9.2.1)]{Abramowitz1984} we know the asymptotic behavior of Bessel functions:
\begin{equation*}\label{eq:bessel_asymptotic}
J_{\nu}(s)=\sqrt{\frac{2}{\pi s}}\cos \left(s-\left(\frac{1}{2}\nu-\frac{1}{4}\right)\pi\right) + O\left(s^{-\frac 3 2}\right).
\end{equation*}
So $\sqrt{p}\int_0^{\infty} h(s)^p \D s$ diverges for $p\to \frac 2 3 $.
By the intermediate value theorem, there is  $p_0 \in \left( \frac 2 3  ,2\right)$ such that 
\begin{equation}\label{eq:nazarov n2 m2}
\sqrt{p_0} \int_0^{\infty} h(s)^{p_0} \D s=\sqrt{p_0}\int_0^{\infty} g(s) ^{p_0} \D s.
\end{equation} 

\subsubsection{Condition (N1)} We investigate the two distribution functions $H$ and $G$.

The distribution function $G$ is given by the inverse of $g$, since $g$ is a decreasing and bijective function $\R_{\geq 0} \to (0,1]$. So  for $y\geq 1$, $G(y)=0$ and for $s\in (0,1)$ we write explicitly
\[
G(y)=\lambda \left(\bigg\{s \Big| y <\exp \left(-\frac{s^2}{8}\right)\bigg\}\right)=\lambda \left(\bigg\{s \Big| s<\sqrt{8 \ln \left(\frac{1}{y}\right)}\bigg\}\right)=\sqrt{8 \ln \left(\frac{1}{y}\right)}.
\]
Its derivative is 
\begin{equation}\label{eq:G prime}
G'(y)=-\frac{\sqrt{2}}{y\ln\left(\frac{1}{y}\right)}.
\end{equation}
Later, we need that $\frac{1}{|G'(y)|}$ is decreasing for $0\leq y \leq \frac{1}{\sqrt{e}}$.

Now we investigate $H$. The function $h$ is oscillating. Denote the $k$-th local maximum of $h$ by $y_k:=\max \{h(s) \mid s\in (s_{k}, s_{k+1})\}$, with $s_k$ the $k$-th zero of the Bessel function $J_1$ and $s_0=0$. The approximation of the first zeros is taken from \cite[p. 748: Table VII]{Watson1966}; $s_1=3.832, s_2=7.016, s_3=10.173$. 

\textbf{Step (i): There is at least one intersection of $G$ and $H$.} \\
From Lemma \ref{lem:bessel small} we know  $h(s)=\left|j_1(s)\right|\leq \exp \left( -\frac{s^2}{8}\right)=g(s)$ for $s \in [0,4]$.  So for $y\geq y_1:$
\begin{align*}
H(y)&=\lambda \left(\{x \in [0,\infty) \mid h(x)>y\}\right) \\
&=\lambda\left(\{x \in [0,s_1] \mid h(x)>y\}\right) \\
&\leq \lambda \left(\{x \in [0,s_1] \mid g(x)>y\}\right) \\
&=G(y). 
\end{align*}
So $G-H\geq0$ for $y\in (y_1,\infty)$. 
Consider (\ref{eq:nazarov n2 m2}) and observe that by Fubini and substitution
\begin{align*}
0&=\int_0^{\infty}(g(s)^{p_0}-h(s)^{p_0})\D s\\
&=\int_0^{\infty} \left(G(y^{\frac{1}{p_0}})-H(y^{\frac{1}{p_0}})\right)\D y\\
&=p_0\int_0^{\infty}y^{p_0-1} \left(G(y)-H(y)\right)\D y.
\end{align*}
So $G-H$ has to change its sign at least once. 

\textbf{Step (ii): There is at most one intersection of $G$ and $H$.}\\
If we prove that $G-H$ is increasing on $(0,y_1)$, this implies $G-H$ changes its sign only once.
We show this by proving that for each interval $(y_{k+1}, y_k)$, the quotient $\frac{|H'|}{|G'|}$ is strictly larger than $1$.  The distribution functions are decreasing, so their derivatives are negative (or $0$). So $\frac{|H'|}{|G'|}>1$ implies $H'<G' $ and therefore $G-H$ is increasing.

\textbf{Step (iii): Estimate the local maxima of $H$.}\\
From \cite[p. 116]{Schafheitlin1908} we know the approximate position of the zeros of the Bessel function $J_1$: 
\begin{equation}\label{eq: zeros of bessel}
s_k \in (k\pi, (k+1/4)\pi).
\end{equation}
In \cite[p. 32]{Koenig2013c} it is noted that the successive maxima of $\left|\sqrt{ \frac {2}{\pi}}\sqrt{s}J_1(s)\right|$ are decreasing to $1$. This implies
\begin{align*}
2\sqrt{\frac {2}{\pi} }\frac{1}{(s_{k+1})^{\frac 3 2}}
\leq y_k 
\leq 2\sqrt{\frac {2}{\pi}}\frac{1}{(s_k)^{\frac 3 2}}.
\end{align*}
In particular, together with (\ref{eq: zeros of bessel}), we get
\begin{align}\label{eq: y k s}
\frac{2 \sqrt{2}}{\pi^2} \frac{1}{(k+\frac{5}{4})^{\frac{3}{2}}} 
\leq y_k 
\leq \frac{2 \sqrt{2}}{\pi^2} \frac{1}{k^{\frac{3}{2}}}.
\end{align}

\textbf{Step (iv): Compute $H$.}\\
For $y\ne y_k$ we claim that
\[
			|H'(y)| = \sum_{s>0, h(s)=y} \frac{1}{|h'(s)|}.
\]
To see this, note that for a bijective function $f$, the distribution function $F$ is given by $F=f^{-1}$ and $F'=\frac{1}{f'}$. Now $H$ can be decomposed into the sum of the bijective parts of $h$, where $H(y)$ is the length of the intervals on the real line, cf. \cite[p. 6]{Nazarov2000}. The equation $h(s)=y$ has one root in $(0,s_1)$ and two roots in each interval $(s_k,s_{k+1})$ for $1 \leq k \leq K$, with some $K\in \mathbb{N}$ depending on $y$.

\textbf{Step (v): Estimate $h'$.}\\
We estimate $h'(s)$ at these roots.
By the recurrence relation for Bessel functions, we have $|h'(s)|=\left(\frac{|2J_1(s)|}{s}\right)'=2\frac{|J_2(s)|}{s}$. 
We approximate $J_2$ with \cite[(8.479)]{Gradshteyn2007} and find $\left| \frac{2J_{2}(s)}{s}\right|\leq 2\sqrt{\frac{2}{\pi}} \frac{1}{\sqrt[4]{s^2-4}}\frac{1}{s}$ for $s\geq 2$. Additionally for $s\geq 3$, $\frac{1}{\sqrt[4]{s^2-4}}\frac{1}{s} \leq \frac{1}{s^{\frac 3 2}}\sqrt{\frac {\pi}{2}}$. So for $s\geq 3$ we estimate 
\begin{equation*}\label{eq:estimate h prime}
|h'(s)| \leq 2\sqrt{\frac{2}{\pi}} \frac{1}{\sqrt[4]{s^2-4}}\frac{1}{s} \leq 2 \frac{1}{s^{\frac 3 2}}.
\end{equation*}
This holds in particular for  $s \in  (s_k, s_{k+1})$, $k\geq 1$, since $s_1\geq 3$.
Therefore \[|h'(s)|\leq 2\frac{1}{s_k^{\frac 3 2}} \leq 2 \frac{1}{(\pi k)^{\frac 3 2}}.\]
For $s\in [0,s_1)$, a rough estimate is sufficient:
\[
|h'(s)| \leq 0.4.
\]

\textbf{Step (vi): Estimate $H'/G'$.}\\
Fix $k$ and let $y\in (y_{k+1},y_k)$. Then 
\begin{align*}
{|H'(y)|} &\geq \left(2.5 + 2 \cdot \frac{1}{2}\pi^{3/2} \sum_{l=1}^k l^{3/2}\right). 
\end{align*}
Since $y_k\leq y_1<\frac{1}{\sqrt{e}}$, we may use (\ref{eq:G prime}) and (\ref{eq: y k s}) to estimate
\begin{align*}
\frac{1}{|G'(y)|}\geq \frac{1}{|G'(y_{k+1})|}\geq \frac{1}{\left|G'\left(\frac{\pi^2}{2\sqrt{2}}\left(k+ \frac{9}{4}\right)^{3/2}\right)\right|}.
\end{align*}
For the quotient we get
\begin{align*}\label{eq:estimate H durch G Strich}
\frac{|H'(y)|}{|G'(y)|} &\geq \left(2.5 + \pi^{3/2} \sum_{l=1}^k l^{3/2}\right) \frac{2}{\pi^2 \left(k+9/4\right)^{3/2}} \sqrt{\ln \left(\frac{\pi^2}{2\sqrt{2}}\left(k+ \frac{9}{4}\right)^{3/2}\right)}=:Q(k).
\end{align*}
We estimate $\sum_{l=1}^{k} l^{\frac{3}{2}}\geq \int_{0}^{k} l^{\frac{3}{2}}\D l=\frac{2}{5}k^{\frac{5}{2}}$. Using this estimate, note that $Q(k)$ is increasing in $k$. By evaluation, $Q(2) >1$, so $Q(k)>1$ for all $k\geq 2$.

Since $Q(1)<1$, the estimate needs to be sharper for $k=1$.
Let $y\in (y_2,y_1)$. The equation $h(s)=y$ has three solutions.
Denote them by $\sigma_1,\sigma_2,\sigma_3$ in ascending order. 
We estimate these roots numerically, using the roots of $h(s)=y_2$. Then we use Lemma \ref{eq:estimatebessellarge} to estimate $|h'|$.
We find $\sigma_1 \in \left(3.3050,s_1 \right)$, so $\frac{1}{|h'(\sigma_1)|}\geq \frac{1}{0.298}$.
And $\sigma_2 \in (4.1896, s_2)$, so $\frac{1}{|h'(\sigma_2)|}\geq \frac{1}{0.199}$, 
as well as  $\sigma_3 \in (4.1896,s_2)$, so $\frac{1}{|h'(\sigma_3)|} > \frac{1}{0.199}$.
The corresponding estimate for $G'$ is $\frac{1}{|G'(y)|}\geq \frac{1}{|G'(y_2)|} \geq 0.077.$
Therefore we get for all $y \in (y_2,y_1)$:
\[\frac{|H'(y)|}{|G'(y)|} >1.\]

Thus we have shown that $\frac{|H'(y)|}{|G'(y)|}>1$ for all $y\in (0,y_1)$.
 
This finishes the proof of condition (N1) and therefore the proof of Theorem \ref{thm_bessel_int} for $m=2$.
{ } \hfill \qedsymbol

\subsection{The case $m\geq 5$}
The previous proof relied on the approximate knowledge of the zeros of the Bessel function. Here we use a different approach. The idea is due to \cite{Brzezinski2011}. The aim is to simplify $j_{\frac{m}{2}}$, use that $j_{\frac m 2}$ decays rapidly, and get rid of the oscillating behavior. Due to the rougher estimates this only works for $m\geq 5$. 
We define
\begin{equation}\label{eq:simplifyBessel}
\tilde j_{\frac{m}{2}}(s):=	
							\begin{cases}
									\left|j_{\frac m 2 } (s)\right|,	& s\in [0,m)\\
									2^{\frac{m+1}{2}}  \frac{\Gamma (\frac m 2 + 1)}{\sqrt{\pi}} \left(s^2-\frac {m^2}{4}\right)^{-\frac 1 4} s^{-\frac m 2}					, & s \in [m, \infty).
							\end{cases}
\end{equation}

For this simplification, by Lemma \ref{eq:estimatebessellarge} it is true that for all $s\geq 0$ 
\begin{equation}\label{eq:bessel_tilde}
j_{\frac{m}{2}}(s) \leq \tilde j_{\frac{m}{2}}(s).
\end{equation}
So it is sufficient to prove the inequality for this simplification of $j_{\frac{m}{2}}$. We apply the Nazarov-Podkorytov lemma \ref{prop:nazarov}.

\subsubsection{Condition (N1)}
We compare $\tilde j_{\frac{m}{2}}(s)$ and $g(s):=\exp \left( -\frac{s^2}{2m+4}\right)$. We claim
\begin{align} \label{eq:n1 a}
\quad \tilde j_{\frac{m}{2}}(s) < g(s), \quad &s\in [0,m], \\ \label{eq:n1 b}
\quad \tilde j_{\frac{m}{2}}(s) > g(s), \quad &s \in (m+2, \infty),\\ \label{eq:n1 c}
\quad \tilde j_{\frac{m}{2}}(s)=g(s), \quad  &\text{for exactly one } s \in (m, m+2). 
\end{align}
Inequality (\ref{eq:n1 a}) corresponds to Lemma \ref{lem:bessel exp}.
Inequality (\ref{eq:n1 b}) is \cite[Lemma 3.19]{Brzezinski2011}; note that the lemma there is also true for $m=5$ by exactly the same argument. 
Property (\ref{eq:n1 c}) is from \cite[Lemma 3.18]{Brzezinski2011}; this does not include  $m=5$ and $m=6$, but one can easily check the statement by hand with analogous arguments.

Since $g$ and $\tilde j_{\frac{m}{2}}$ are bounded by $1$, for $y\geq 1$ we have $G(y)=0=\tilde J_{\frac m 2}(y)$, where $\tilde J_{\frac m 2 }$ is the distribution function of $\tilde j _{\frac m 2 }$. The functions $g$ and $\tilde j_ {\frac m 2 }$ intersect exactly once, so the difference of the cumulative distribution functions changes its sign exactly once as well. This shows (N1).

\subsubsection{Condition (N2)} We will show
\begin{align}
&\text{for } p\to \frac {2}{m+1}, \quad  \sqrt{p} \int_0^{\infty} \tilde j_{\frac{m}{2}}(s)^{p} \D s   \longrightarrow \infty,  
\label{eq:n2 a}\\
&\sqrt{2} \int_0^{\infty} \tilde j_{\frac{m}{2}}(s)^{2} \D s < \sqrt{\pi}\sqrt{\frac m 2  +1}, \label{eq:n2 b}\\
&\exists p_0 \in \left(\frac{2}{m+1},2\right] \colon \sqrt{p_0} \int_0^{\infty} \tilde j_{\frac{m}{2}}(s)^{p_0} \D s =  \sqrt{p_0}  \int_0^{\infty} g(s)^{p_0}\D s= \sqrt{\pi}{\sqrt{\frac m 2 +1}}. \label{eq:n2 c}
\end{align}
For large $s$ the function $\tilde j_{\frac{m}{2}}$ is asymptotically equal to $\frac{2^{\frac{m+1}{2}}}{\sqrt{\pi}} \Gamma\left(\frac m 2 + 1\right) s^{-\frac{m+1}{2}}$. 
Therefore $\tilde j_{\frac{m}{2}}(\cdot)^p$ is integrable for $p>\frac {2}{m+1}$, and $\int_0^{\infty} \tilde j_{\frac{m}{2}}(s)^p\D s$ diverges for $p\to \frac{2}{m+1}$; this is (\ref{eq:n2 a}).

For inequality (\ref{eq:n2 b}) evaluate the integral. We have
\begin{align*}
				&\sqrt{2} \int_0^{\infty} \tilde j_{\frac{m}{2}}(s)^2 \D s \\
				=& \sqrt{2}\int_0^m \left|j_{\frac m 2} (s)\right|^2 \D s + \sqrt{2}\int_m^{\infty} \left( 2^{\frac{m+1}{2}}  \frac{\Gamma (\frac m 2 + 1)}{\sqrt{\pi}} \left(s^2-\frac {m^2}{4}\right)^{-\frac 1 4} s^{-\frac m 2}\right)^2 \D s \\
				\leq& \sqrt{2}\int_0^{\infty} \left|j_{\frac m 2} (s)\right|^2 \D s + \sqrt{2}\int_m^{\infty} \left( 2^{\frac{m+1}{2}}  \frac{\Gamma (\frac m 2 + 1)}{\sqrt{\pi}} \left(s^2-\frac {m^2}{4}\right)^{-\frac 1 4} s^{-\frac m 2}\right)^2 \D s \\
				=& \sqrt{2} \int_0^{\infty} 2^m \Gamma \left(\frac{m}{2} + 1\right)^2 \frac{J_{\frac{m}{2}} (s) ^2}{s^m} \D s
							+  2^{m+\frac{3}{2}} \frac{ \Gamma\left(\frac{m}{2}+1\right)^2}{\pi} \int_m^{\infty} \left(s^2-\frac{m^2}{4}\right)^{-\frac{1}{2}} s^{-m} \D s.
\end{align*}

The first integral is evaluated by \cite[6.575 (2)]{Gradshteyn2007} and then estimated by \mbox{Lemma \ref{lem:Gammafunctions}:}
\begin{align*}
	\sqrt{2} \int_0^{\infty} 2^m \Gamma \left(\frac{m}{2} + 1\right)^2 \frac{J_{\frac{m}{2}} (s) ^2}{s^m} \D s &= \sqrt{2}\sqrt{\pi}\frac{\Gamma \left(\frac{m}{2} +1\right)^2 \Gamma (m)}{\Gamma \left(m+\frac{1}{2} \right)\Gamma\left(\frac{m}{2}+\frac{1}{2}\right)^2} \\
	&\leq \sqrt{\pi} \frac{m+2}{m+1}\frac{\sqrt{m}}{\sqrt{2}}.
\end{align*}

For the second summand, we estimate the integrand by $\left(s^2-\frac{m^2}{4}\right)^{-\frac{1}{2}} s^{-m} \leq \sqrt{\frac{4}{3}} s^{-m-1}$, which is true for $s\geq m$. Then use again Stirling's formula for $\Gamma \left(\frac{m}{2} + 1\right)^2$ and get
\begin{align*}
2^{m+\frac{3}{2}} \frac{ \Gamma\left(\frac{m}{2}+1\right)^2}{\pi} \int_m^{\infty} \left(s^2-\frac{m^2}{4}\right)^{-\frac{1}{2}} s^{-m} \D s
&\leq 2^{m+\frac{3}{2}} \frac{ \Gamma\left(\frac{m}{2}+1\right)^2}{\pi}\sqrt{\frac 4 3} \int_m^{\infty} s^{-m-1} \D s\\
 &= 2^{m+\frac{3}{2}} \sqrt{\frac{4}{3}}\frac{ \Gamma\left(\frac{m}{2}+1\right)^2}{\pi} m^{-(m+1)} \\
&\leq \frac{\sqrt{2}^5}{\sqrt{3}}\exp\left(\frac{1}{3m}\right)\exp(-m) .
\end{align*}
It remains to show 
\begin{align}
\sqrt{\pi}\frac{m+2}{m+1}\frac{\sqrt{m}}{\sqrt{2}}+\frac{\sqrt{2}^5}{\sqrt{3}}\exp\left(\frac{1}{3m}\right)\exp(-m) \leq \sqrt{\pi}\sqrt{\frac{m}{2}+1}.
\end{align}
This follows if we prove the stronger inequality
\begin{equation}\label{eq:m groesser 5}
3.65\exp(-m)\leq \sqrt{\pi}\left( \sqrt{\frac m 2 +1}-\frac{m+2}{m+1}\frac{\sqrt{m}}{2}\right).
\end{equation}
Note that $\exp(m)\left( \sqrt{m+2}-\frac{m+2}{m+1}\sqrt{m}\right) \geq \exp(m)\frac{1}{3}m^{-\frac 3 2}$, and $\exp(m)\frac{1}{3}m^{-\frac 3 2}$ is increasing in $m$. For $m=5$, inequality (\ref{eq:m groesser 5}) is true, and so it is true for all $m\geq 5$. This proves (\ref{eq:n2 b}).

Now (\ref{eq:n2 c}) follows by the intermediate value theorem.

Thus we proved (N1) and (N2), so the Nazarov-Podkorytov lemma gives the desired result.\hfill \qedsymbol

\subsection{The case $m\in \{3,4\}$}\label{sec:m34}
The estimates made above by the simplification of $j_{\frac{m}{2}}$ in (\ref{eq:simplifyBessel}) are too rough for $m<5$, since $j_{\frac m 2}$ decreases too slowly for them to work. So we need a different approach here that involves numerical estimates. Therefore one has to treat the cases $m\in\{3,4\}$ separately. The idea is basically given in \cite{Brzezinski2011}, and it is a generalization of \cite{Oleszkiewicz2000}. This approach also works for $m=2$.

With Lemma \ref{lem:estimateintegral}, we prove the original integral inequality from Theorem \ref{thm_bessel_int} for $m\in\{3,4\}$. Split the integral into two parts and estimate them separately: 
\[
\int_0^{\infty} \left|j_{\frac{m}{2}}(s)\right|^p \D s =\int_0^{\frac{m}{2}+3} \left|j_{\frac{m}{2}}(s)\right|^p \D s +\int_{\frac{m}{2}+3}^{\infty} \left|j_{\frac{m}{2}}(s)\right|^p \D s
\]
For the first integral, we use the pointwise estimate Lemma \ref{lem:bessel small} and then \mbox{Lemma \ref{lem:estimateintegral}.} 
\begin{equation}\label{eq:integral part 1}  
\int_0^{\frac{m}{2}+3} \left|j_{\frac{m}{2}}(s)\right|^p \D s \leq  \frac{\sqrt{\pi}}{\sqrt{p}} \sqrt{\frac{m}{2}+1}\left(1-\frac{3}{4}\frac{1}{p(m+4)}+\frac{105}{16}\frac{1}{2p^2(m+4)^2} \right)
\end{equation}
For the second integral, we estimate the integrand pointwise by Lemma \ref{lem:besselpointwiseestimate}. This gives 
\begin{align}\label{eq:integral part 2}
&\int\limits_{\frac{m}{2}+3}^{\infty} \left|j_{\frac{m}{2}}(s)\right|^p \D s \\ \nonumber
&\leq \left( 2^{\frac{m+1}{2}} \frac{\Gamma(\frac{m}{2}+1)}{\sqrt{\pi}} \frac{\sqrt{m+6}}{\sqrt[4]{12m+36}}\right)^p \int_{\frac{m}{2}+3}^{\infty} s^{-\frac{m+1}{2}p} \D s \\ \nonumber 
&=\left(2^{\frac{m+1}{2}} \frac{\Gamma(\frac{m}{2}+1)}{\sqrt{\pi}} \frac{\sqrt{m+6}}{\sqrt[4]{12m+36}}\right)^p \frac{2}{(m+1)p-2}\left(\frac{m}{2}+3\right)^{1-\frac{m+1}{2}p}.
\end{align}
With the estimates (\ref{eq:integral part 1}) and (\ref{eq:integral part 2}) of the two parts of the integral, it remains to prove the following inequality for $p\geq 2$ and $m\in \{3,4\}$
\begin{align*}
 &\phantom{\leq}\frac{\sqrt{\pi}}{\sqrt{p}} \sqrt{\frac{m}{2}+1}\left(1- \frac{3}{4}\frac{1}{p(m+4)} +\frac{105}{16}\frac{1}{2p^2(m+4)^2} \right) 
\\ &\phantom{\leq}+
\left(2^{\frac{m+1}{2}} \frac{\Gamma(\frac{m}{2}+1)}{\sqrt{\pi}} \frac{\sqrt{m+6}}{\sqrt[4]{12m+36}}\right)^p \frac{2}{(m+1)p-2}\left(\frac{m}{2}+3\right)^{1-\frac{m+1}{2}p}
\\ &\leq \frac{\sqrt{\pi}}{\sqrt{p}} \sqrt{\frac{m}{2}+1}.
\end{align*}
Subtract $\frac{\sqrt{\pi}}{\sqrt{p}} \sqrt{\frac{m}{2}+1}$ from both sides. For $m=3$ this reads as
\begin{align*}
\frac{\sqrt{\frac 5 2 \pi}}{\sqrt{p}}\left(\frac{15}{224p^2}-\frac{3}{28p}\right)+\left(\frac{9}{\sqrt[4]{2}\sqrt{6}}\right)^p\frac{2}{4p-2} \left(\frac{9}{2}\right)^{1-2p}\leq 0.
\end{align*}
Multiplying by $p^{\frac 5 2}(4p-2)$ and simplifying,  this reduces to show
\begin{align}\label{eq:last summand}
-p^2 \frac 3 7 \sqrt{\frac 2 5 \pi} + p\frac{27}{56} \sqrt{\frac 2 5 \pi} - \frac{30}{224} \sqrt{\frac 2 5 \pi} + \left(\frac {4}{9\sqrt[4]{2}\sqrt{6}}\right)^p9p^{\frac 5 2}\leq 0.
\end{align}
The last summand of the left-hand side of (\ref{eq:last summand}) is decreasing in $p$ for $p\geq 2$ and its value for $p=2$ is less than $\frac{ 32}{27}$. So we estimate the left-hand side of (\ref{eq:last summand}) by a quadratic function and get
\begin{align*}
&\phantom{=}-p^2 \frac 3 7 \sqrt{\frac 2 5 \pi} + p\frac{27}{56} \sqrt{\frac 2 5 \pi} - \frac{30}{224} \sqrt{\frac 2 5 \pi} + \left(\frac {4}{9\sqrt[4]{2}\sqrt{6}}\right)^p9p^{\frac 5 2} \\ 
&\leq -p^2 \frac 3 7 \sqrt{\frac 2 5 \pi} + p\frac{27}{56} \sqrt{\frac 2 5 \pi} - \frac{30}{224} \sqrt{\frac 2 5 \pi} + \frac{32}{27}.
\end{align*}
This function has its maximum in $p=\frac{9}{16}$, so it is decreasing for $p\geq 2$. For $p=2$ the value is  $-\frac{99}{224}\sqrt{10}\sqrt{\pi}+32/27<0$.
This proves the inequality. \\
\indent This argument works analogously for $m=4$. \hfill \qedsymbol

\section*{Acknowledgements} This work is part of my PhD Thesis. I thank my advisor Hermann K{\"o}nig for his support and advice.
My research was partly supported by DFG (project KO 962/10-1).

\end{document}